\documentclass [14pt]{amsart}
\usepackage{amsmath}
\usepackage{amssymb}
\usepackage{amscd}
\usepackage{epsfig}
\usepackage{amsxtra}
\usepackage{pifont}
\usepackage{mathabx}
\usepackage{MnSymbol}
\usepackage{accents}
\usepackage{scalerel,stackengine}
\stackMath
\newcommand\reallywidehat[1]{%
\savestack{\tmpbox}{\stretchto{%
  \scaleto{%
    \scalerel*[\widthof{\ensuremath{#1}}]{\kern-.6pt\bigwedge\kern-.6pt}%
    {\rule[-\textheight/2]{1ex}{\textheight}}
  }{\textheight}%
}{0.5ex}}%
\stackon[1pt]{#1}{\tmpbox}%
}

\newcommand\reallywidecheck[1]{%
\savestack{\tmpbox}{\stretchto{%
  \scaleto{%
    \scalerel*[\widthof{\ensuremath{#1}}]{\kern-.6pt\bigwedge\kern-.6pt}%
    {\rule[-\textheight/2]{1ex}{\textheight}}
  }{\textheight}%
}{0.5ex}}%
\stackon[1pt]{#1}{\scalebox{-1}{\tmpbox}}%
}

\numberwithin{equation}{section}

\newcommand{\supp}{\mbox{\rm supp}}
\newcommand{\dens}{\mbox{\rm dens}}

\newcommand{\RR}{{\mathbb R}}

\newcommand{\CC}{{\mathbb C}}

\newcommand{\vol}{\mbox{\rm vol}}

\newcommand{\cL}{{\mathcal L}}

\newcommand{\hG}{\widehat{G}}

\newcommand{\oplam}{\mbox{\Large $\curlywedge$}}
\newcommand{\sL}{\mathsf{L}}
\newcommand{\Gd}{G_{\mathsf{d}}}
\newcommand{\cF}{{\mathcal F}}
\newcommand{\cM}{{\mathcal M}}
\newcommand{\cS}{{\mathcal S}}

\newcommand{\Cu}{C_{\mathsf{u}}}
\newcommand{\Cc}{C_{\mathsf{c}}}
\newcommand{\Cz}{C^{}_{0}}
\newcommand{\Kt}{K_2}

\newcommand{\WAP}{\mathcal{WAP}}
\newcommand{\SAP}{\mathcal{SAP}}

 \newtheorem{theorem}{Theorem}[section]
 \newtheorem{lemma}[theorem]{Lemma}
 
 \newtheorem{corollary}[theorem]{Corollary}

 \newtheorem{definition}[theorem]{Definition}
 
  \newtheorem{remark}[theorem]{Remark}

\begin{document}
\title{On the Fourier Transformability of Strongly Almost Periodic Measures }

\author{Nicolae Strungaru}
\address{Department of Mathematical Sciences, MacEwan University \\
10700 -- 104 Avenue, Edmonton, AB, T5J 4S2\\
and \\
Department of Mathematics\\
Trent University \\
Peterborough, ON
and \\
Institute of Mathematics ``Simon Stoilow''\\
Bucharest, Romania}
\email{strungarun@macewan.ca}
\urladdr{http://academic.macewan.ca/strungarun/}

\begin{abstract} In this paper we characterize the Fourier transformability of strongly almost periodic measures in terms of an integrability condition for its Fourier Bohr series. We also provide a necessary and sufficient condition for a strongly almost periodic measure to be a Fourier transform of a measure.
We discuss the Fourier transformability of a measure on $\RR^d$ in terms of its Fourier transform as a tempered distribution. We conclude by looking at a large class of such measures coming from the cut and project formalism.
\end{abstract}

\maketitle

\section{Introduction}

The Fourier transform of functions plays a fundamental role in many areas of mathematics. In the first half of the $20^{\rm th}$ century, Laurent Schwartz extended the Fourier transform to a larger class of objects, namely tempered distribution. This theory extends the classical Fourier transform of functions, and includes all finite measures, all continuous and bounded functions as well as a large class of unbounded measures. Some of the notions have been extended to arbitrary locally compact abelian groups (LCAG's) $G$  \cite{BRU}, but so far these extensions are not as useful for the study of measures as in the case $G=\RR^d$.

Motivated by Bochner's Theorem, Argabright and deLamadrid introduced the notion of Fourier transform for unbounded measures over arbitrary locally compact Abelian groups (LCAG's), and proved that positive definite measures are Fourier transformable \cite{ARMA1} (see also \cite{BF,MoSt}). Their theory of Fourier transform of measures generalizes the classical theory of Fourier transform of functions, as well as the Fourier-Stieltjes transform. The Fourier transform of measures plays a fundamental role for mathematical diffraction and aperiodic order(see for example \cite{TAO,BL3,LAG,DL03,MoSt,CRS,CRS2,NS11}).

There is a hidden strong connection between the Fourier transform of measures and the class of (weakly) almost periodic functions and measures. Eberlein proved that there exists a canonical decomposition of a weakly almost periodic function into a strongly almost periodic function and a null weakly almost periodic function \cite{EBE3}. We will refer to this decomposition as the \textbf{Eberlein decomposition}. Positive definite continuous functions, and hence the Fourier transform of any finite measure, are weakly almost periodic \cite{EBE}. Given a finite measure, $\mu$, the Eberlein decomposition
of the weakly almost periodic function $\widehat{\mu}$ is exactly the Fourier transform \cite{EBE,MoSt} of the Lebesgue decomposition
\begin{displaymath}
  \mu=\mu_{\mathsf{pp}}+\mu_{\mathsf{c}} \,.
\end{displaymath}

GildeLamadrid and Argabright extended the concept of almost periodicity to translation bounded measures, via convolution with compactly supported continuous functions \cite{ARMA} (see also \cite{MoSt} for a self contained exposition of these topics). They showed that the weakly almost periodic measures also have a canonical Eberlein decomposition. Moreover, the Fourier transform $\widehat{\mu}$ of each transformable measure $\mu$ is weakly almost periodic, and the Eberlein decomposition of the Fourier transform is exactly the dual of the Lebesgue decomposition of $\mu$ \cite{ARMA}.
Recently, the Fourier dual of this result was proven by Moody and I \cite{MoSt}: if any translation bounded Fourier transformable measure $\mu$ is weakly almost periodic, the strong almost periodic component $\mu_{\mathsf{s}}$ and the null weakly almost periodic component $\mu_0$ are Fourier transformable, and their Fourier transforms are exactly the pure point component $\widehat{\mu}_{\mathsf{pp}}$ and the continuous component $\widehat{\mu}_{\mathsf{c}}$ of $\widehat{\mu}$. This last version of the result is important for mathematical diffraction, since we would like to study the pure point spectrum $\widehat{\gamma}_{\mathsf{pp}}$ and the continuous spectrum $\widehat{\gamma}_{\mathsf{pp}}$ of a structure $\omega$, without going to the Fourier dual space. These results allow us to study the pure point and continuous spectra, respectively, by studying the components $\gamma_{\mathsf{s}}$ and $\gamma_0$, respectively, of the autocorrelation $\gamma$ of $\omega$, an idea which was used effectively in many places, (such as \cite{JBA,TAO,LR,NS1,NS2,NS4,NS5}, to name a few). The particular connection between strong almost periodicity and pure point Fourier transform was also exploited in articles such as \cite{TAO,BL,BL2,BL3,BLM,BM,FAV,Fa15,LR,LS,LS2,MS,CR,CRS,CRS2,NS3,NS11}.

It follows from the results in \cite{ARMA} that if a measure $\mu$ is Fourier transformable, its Fourier transform $\widehat{\mu}$ is strongly almost periodic exactly when $\mu$ is a pure point measure. In this case, the strongly almost periodic measures $\widehat{\mu}$ has a Fourier Bohr series (see Definition~\ref{defi:FB series} below) $\cF_{\mathsf{d}}(\widehat{\mu})$, which is exactly the reflection $\mu^\dagger$ of $\mu$. Same way, if $\mu$ is Fourier transformable, its Fourier transform $\widehat{\mu}$ is pure point exactly when $\mu$ is strongly almost periodic, and $\widehat{\mu}$ is exactly the Fourier-Bohr series $\cF_{\mathsf{d}}(\mu)$.

Every strongly almost periodic measure $\mu$ comes with a Fourier-Bohr series $\cF_{\mathsf{d}}(\mu)$, which is exactly $\widehat{\mu}$ (respectively $\widecheck{\mu}^\dagger$) whenever when $\mu$ is Fourier transformable (or a Fourier transform). It is natural to ask what extra condition should $\cF_{\mathsf{d}}(\mu)$ satisfy in order for $\mu$ to be Fourier transformable (respectively a Fourier transform). The main goal of this paper is to answer to these two questions.

We show in Theorem~\ref{thm: main result} that a necessary and sufficient condition for a strongly almost periodicity measure $\mu$  to be Fourier transformable is a certain integrability condition, which we call weak admissibility (see Defi.~\ref{defi:admiss}, Defi.~\ref{defi:admiss formal sums} below) being satisfied by the Fourier Bohr series. The second question is answered in Theorem~\ref{thm: sap as FT}: we show that a strongly almost periodic measure $\mu$ is a Fourier transform if and only if $\mu$ is weakly admissible and its Fourier Bohr series is a measure.

In the particular case $G=\RR^d$, which is the case in most of the practical applications, we use a result of Lin \cite{lin} to show that the weak admissibility condition can be replaced by the much more concrete notion of translation boundedness. As a consequence, we get that a strongly almost periodic measure $\mu \in \SAP(\RR^d)$ is Fourier transformable if and only if its Fourier Bohr series is a translation bounded measure.
Same way, a strongly almost periodic measure $\mu \in \SAP(\RR^d)$ is a Fourier transform if and only if its Fourier Bohr series is a measure.

We also study the connection between the Fourier transformability of a measure in $\RR^d$ and its Fourier transformability as a tempered distribution.
In \cite{ARMA1}, the authors introduced a measure $\mu$ on $\RR$, which is positive definite, tempered as a distribution, but for which the variation measure $|\mu|$ is not tempered. In particular $\mu$ is not translation bounded as a measure. Since $\mu$ is positive definite, it is Fourier transformable and its Fourier transform $\widehat{\mu}=: \nu$ is translation bounded \cite{ARMA1}. It follows that $\nu$ is a tempered distribution, whose Fourier transform is the measure $\mu^\dagger$, but which is not Fourier transformable as a measure (see \cite{ST} for more details). This raises an interesting question: what is the connection between the Fourier transform of measures on $\RR^d$ and their Fourier transform as distributions (compare \cite{ST}). We answer this question in Theorem \ref{thm: FT measure and distribution}: we show that a translation bounded measure $\mu$ on $\RR^d$ is Fourier transformable as a measure if and only if its Fourier transform in tempered distribution sense is a translation bounded measure. Moreover, in this case, the two Fourier transforms coincide.

\section{Definitions and notations}

Throughout this paper $G$ will denote a locally compact abelian group (LCAG). We will denote by $\Cu(G)$ the space of uniformly continuous and bounded functions on $G$. $\Cz(G)$ and $\Cc(G)$ will denote the subspaces of $\Cu(G)$ consisting of functions vanishing at infinity, and functions with compact support respectively.

\smallskip

In the spirit of Bourbacki \cite{BOURB}, by a \textbf{measure} we understand a linear function $\mu : \Cc(G) \to \CC$ such that, for each compact set $K \subset G$ there exists a constant $C_K$ such that, for all $f \in \Cc(G)$ with $\supp(f) \subset K$ we have
\begin{displaymath}
\left| \mu(f) \right| \leq C_k \| f \|_\infty \,.
\end{displaymath}
The equivalence between this definition and the measure theory definition of regular Radon measures is provided by the Riesz-Representation Theorem \cite{Reiter,ReiterSte} (see also \cite[Appendix]{CRS2} for a discussion of this).

We will use often $\langle \mu, f\rangle$ or $\int_{G} f(t) d \mu(t)$ instead of $\mu(f)$.

\smallskip

Next, let us recall the definition of Fourier transformability for measures.

\begin{definition} A measure $\mu$ is called \textbf{Fourier transformable} if there exists some measure $\widehat{\mu}$ on $\widehat{G}$ such that, for all $f \in \Cc(G)$ we have $\left|\widecheck{f}\right|^2 \in L^1(|\widehat{\mu|})$ and
\begin{displaymath}
\langle \mu, f*\widetilde{f} \rangle = \langle\widehat{\mu}, \left|\widecheck{f}\right|^2\rangle \,.
\end{displaymath}
\end{definition}

\smallskip
In the spirit of \cite{ARMA} we define
\begin{displaymath}
\Kt(G):= \mbox{Span} \{ f*g | f, g \in \Cc(G) \} \,.
\end{displaymath}
Given a subspace $V \subset L^1(G)$ we will denote by
\begin{displaymath}
\reallywidehat{V}:= \{ \widehat{f} | f \in V \} \subset \Cz(\widehat{G}) \,
\end{displaymath}
\begin{remark}
\begin{itemize}
\item[(i)] By the depolarisation identity, a measure is Fourier transformable if and only if there exists some measure $\widehat{\mu}$ on $\widehat{G}$ such that, $\widehat{\Kt(G)} \subset L^1(|\widehat{\mu|})$ and for all $f \in \Kt(G)$ we have
\begin{displaymath}
\langle \mu, f \rangle = \langle\widehat{\mu}, \widecheck{f}\rangle \,.
\end{displaymath}
\item[(ii)] Any positive definite measure is Fourier transformable and its transform is positive \cite{ARMA1,BF}.
\end{itemize}
\end{remark}

\smallskip

Next, recall that for any measure $\mu$, there exists \cite{Ped,ReiterSte} a positive measure $|\mu|$, called the \textbf{variation of $\mu$} such that, for all $f \in \Cc(G)$ with $f \geq 0$ we have
\begin{displaymath}
|\mu|(f) = \sup \{ \left| \mu(g) \right| \, | \, g \in \Cc(G), |g| \leq f \} \,.
\end{displaymath}
For details about the existence of the variation measure we refer the reader to \cite[Appendix]{CRS2}.

\smallskip

Let us recall now the definition of translation boundedness.

\begin{definition} A measure $\mu$ is called \textbf{translation bounded} if for all compact sets $K \subset G$ we have
\begin{displaymath}
\| \mu \|_K := \sup_{x \in G} \left| \mu \right|(x+K) < \infty \,.
\end{displaymath}
We denote the space of translation bounded measures by $\cM^\infty(G)$.
\end{definition}

\begin{remark}
\begin{itemize}
\item [(i)] A measure $\mu$ is translation bounded if and only if
\begin{displaymath}
\| \mu \|_K < \infty \,,
\end{displaymath}
for one compact set $K$ with non-empty interior \cite{BM}.
\item [(ii)] If $K$ is a fixed compact set with non-empty interior, then $\| \, \|_K$ is a norm on $\cM^\infty(G)$.
\end{itemize}
\end{remark}

An alternate characterisation of translation boundedness is given by the following result:
\begin{theorem}\cite[Thm.~1.1]{ARMA1} A measure $\mu$ is translation bounded if and only if for all $f \in \Cc(G)$ we have $\mu*f \in \Cu(G)$.
\end{theorem}

\subsection{Almost Periodic Measures}

In this subsection we review briefly the basic properties of almost periodic functions and measures. For a more detailed review of this we refer the reader to \cite{MoSt}.

\begin{definition}

A function $f \in \Cu(G)$ is called \textbf{strong almost periodic} or \textbf{Bohr almost periodic} if the closure with respect to $\| \, \|_\infty$ is compact in $(\Cu(G), \| \, \|_\infty)$.

A function $f \in \Cu(G)$ is called \textbf{weakly almost periodic} if the closure with respect to the weak topology of the Banach space $(\Cu(G), \| \, \|_\infty)$ is weakly-compact.

We denote the spaces of strong respectively weakly almost periodic functions by $SAP(G)$ respectively $WAP(G)$.
\end{definition}

\smallskip

\begin{remark}
\begin{itemize}
\item[(i)] $WAP(G)$ and $SAP(G)$ are closed subspaces of $(\Cu(G), \| \, \|_\infty)$ \cite{EBE} (see also \cite{MoSt}).
\item[(ii)] $WAP(G)$ and $SAP(G)$ are closed under multiplication, complex conjugation, reflection and taking the absolute value \cite{EBE}.
\end{itemize}
\end{remark}

\smallskip
Next, we review the notion of null weakly almost periodicity for functions. We first need to recall the definition of the mean of a weakly almost periodic function.

\begin{lemma}\cite{EBE,MoSt} Let $f \in WAP(G)$ and let $\{ A_n \}$ be a F\"olner sequence in $G$. Then, the limit
\begin{displaymath}
\lim_{n} \frac{1}{\vol(A_n)} \int_{x+A_n} f(t) dt
\end{displaymath}
exists uniformly in $x \in G$, and is independent of $x$ and of the choice of the F\"olner sequence $\{A_n \}$.
\end{lemma}

\begin{definition} Let $f \in \WAP(G)$ and let $\{ A_n \}$ be a F\"olner sequence in $G$. The number
\begin{displaymath}
M(f):= \lim_{n} \frac{1}{\vol(A_n)} \int_{A_n} f(t) dt  \,,
\end{displaymath}
is called the \textbf{mean of $f$}.

A function $f \in WAP(G)$ is called \textbf{null weakly almost periodic} if $M(|f|)=0$. We denote the space of null weakly almost periodic functions by $WAP_0(G)$.
\end{definition}

\bigskip

In the spirit of \cite{ARMA} we extend the notions of almost periodicity to measures (see also \cite{MoSt}).

\begin{definition} A measure $\mu \in \cM^\infty(G)$ is called \textbf{strong almost periodic}, \textbf{weakly almost periodic} and \textbf{null weakly almost periodic}, respectively, if for all $f \in \Cc(G)$ the function
$f*\mu$ is strong almost periodic, weakly almost periodic respectively null weakly almost periodic.

We will denote the spaces of almost periodic measures by $\SAP(G), \WAP(G)$ respectively $\WAP_0(G)$.
\end{definition}

Similar to functions, weakly almost periodic measures have a well defined mean:

\begin{lemma} \cite{ARMA,MoSt} Let $\mu \in \WAP(G)$. Then, there exists a number $M(\mu)$ such that, for all $f \in \Cc(G)$ we have
\begin{displaymath}
M(\mu*f) = M(\mu) \int_{G} f(t) dt \,.
\end{displaymath}
Moreover, if $\{ A_n \}$ is any van Hove sequence in $G$, we have
\begin{displaymath}
M(\mu)= \lim_{n} \frac{\mu(x+A_n)}{\vol(A_n)} \,,
\end{displaymath}
uniformly in $x$.
\end{lemma}

\smallskip

As proven by Eberlein for functions \cite{EBE3}, and Argabright and deLamadrid for measures \cite{ARMA}, the space $\SAP(G)$ is a direct summand in $\WAP(G)$ and $\WAP_0(G)$ is its complement.

\begin{theorem} \cite{ARMA}
\begin{displaymath}
\WAP(G) = \SAP(G) \bigoplus \WAP_0(G) \,.
\end{displaymath}
In particular, every measure $\mu \in \WAP(G)$ can be written uniquely
\begin{displaymath}
\mu=\mu_{\mathsf{s}}+ \mu_0 \,,
\end{displaymath}
with $\mu_{\mathsf{s}} \in \SAP(G), \mu_0 \in \WAP_0(G)$. We will refer to this as the \textbf{Eberlein decomposition of $\mu$}.
\end{theorem}

For Fourier transformable measures the Eberlein decomposition is the Fourier dual of the Lebesgue decomposition into pure point and continuous components \cite{ARMA,MoSt}.

\smallskip

We complete the section by reviewing the Eberlein convolution.

\begin{theorem} \cite{EBE,EBE2} If $f,g \in WAP(G)$ then
\begin{displaymath}
f \circledast g(t) = M_x ( f(t-x)g(x)) \,,
\end{displaymath}
is well defined and belongs to $SAP(G)$.

We will call $f \circledast g$ the \textbf{Eberlein convolution of $f$ and $g$}.
\end{theorem}

\begin{theorem} \cite{ARMA} If $f \in SAP(G)$ and $\mu \in \WAP(G)$ then
\begin{displaymath}
f \circledast \mu(t) = M ( f(t- \cdot ) \mu) \,,
\end{displaymath}
is well defined and belongs to $\SAP(G)$.

We will call $f \circledast \mu$ the \textbf{Eberlein convolution of $f$ and $\mu$}.
\end{theorem}

Recently, the notion of Eberlein convolution was extended to two weakly almost periodic measures in \cite{LS2}.

\smallskip

Finally, we review the notion of approximate identity for the Eberlein convolution.

\begin{definition} A net $\{ f_\alpha \}$ with $f_\alpha$ is \textbf{an approximate identity for $(SAP(G), \circledast)$} if for all $f \in SAP(G)$ we have
\begin{displaymath}
f= \lim_\alpha f \circledast f_\alpha
\end{displaymath}
in $(SAP(G), \| \, \|_\infty)$.
\end{definition}

\begin{remark}
\begin{itemize}
\item[(i)] Consider the natural embedding $G \hookrightarrow G_{\mathsf{b}}$ of $G$ into its Bohr compactification.

Then $f_\alpha$ is an approximate identity for $(SAP(G), \circledast)$ if and only if there exists an approximate identity $g_\alpha$ for $(C(G_{\mathsf{b}}), *)$ such that $f_\alpha$ is the restriction to $G$ of $g_\alpha$ \cite{EBE,ARMA,MoSt}. Moreover
\begin{displaymath}
M(f_\alpha)=\int_{G_{\mathsf{b}}} g_{\mathsf{b}}(s) d \theta_{G_{\mathsf{b}}}(s) \,.
\end{displaymath}
In particular, approximate identities for $(SAP(G), \circledast)$ exist, and can be chosen such that $f_\alpha \geq 0, f_\alpha(-x)=f_\alpha(x)$ and $M(f_\alpha)=1$.

\item[(ii)] If $f_\alpha$ is an approximate identity for $(SAP(G), \circledast)$ then
\begin{displaymath}
\lim_\alpha M(f_\alpha)=1 \,.
\end{displaymath}
\end{itemize}
\end{remark}

\section{Weakly admissible Measures}

In this section we introduce a new concept for a measure, which we will call weakly admissible (compare to the definition of admissible measures \cite{lin}), and study the basic properties of weakly admissible measures.

The definition of weak admissibility is simply the integrability condition from the definition of Fourier transformability, and its importance to the Fourier theory for measures is emphasized by \cite[Thm.~3.10]{CRS2}(Theorem~\ref{thm: twice FT} below).

\begin{definition}\label{defi:admiss} A measure $\mu \in \cM(G)$ is called \textbf{weakly admissible} if we have
\begin{displaymath}
\reallywidehat{K_2(\widehat{G})} \subset L^1(|\mu|) \,.
\end{displaymath}
\end{definition}

Note that $\mu \in \cM(G)$ is weakly admissible if and only if
\begin{displaymath}
\reallywidehat{\Cc(\widehat{G})} \subset L^2(|\mu|) \,.
\end{displaymath}
\smallskip

We start by stating a simple lemma which contains few straightforward properties of weakly admissible measures.

\begin{lemma}\label{lemma-basic prop}
\begin{itemize}
\item[(i)] $\mu$ is weakly admissible if and only if $|\mu|$ is weakly admissible.
\item[(ii)] If $\mu$ is weakly admissible and $|\nu| \leq |\mu|$ then $\nu$ is weakly admissible.
\item[(iii)] $\mu$ is weakly admissible if and only if $\mu_{\mathsf{pp}}, \mu_{ac}$ and $\mu_{sc}$ are weakly admissible.
\item[(iv)] If $\mu$ is Fourier transformable, then $\widehat{\mu}$ is weakly admissible.
\item[(v)] If $\mu$ is weakly admissible then $\overline{\mu}, \widetilde{\mu}, \mu^\dagger$ and $T_t \mu$ are weakly admissible.
\item[(vi)] If $\mu$ is weakly admissible and $f \in \Cu(G)$ then $f \mu$ is weakly admissible.
\end{itemize}
\end{lemma}
\begin{proof}

(i) is obvious by the definition of weak admissibility.

(ii) If $f \in \Cc(\widehat{G})$ then $\left|\widehat{f}\right|^2$ is continuous, hence measurable. Moreover
\begin{displaymath}
\int_{G} \left|\widehat{f}\right|^2 d |\nu| \leq \int_{G} \left|\widehat{f}\right|^2 d |\mu| < \infty \,.
\end{displaymath}
This shows that $\reallywidehat{\Cc(\widehat{G})} \subset L^2(|\nu|)$.

(iii) Follows immediately from
\begin{displaymath}
\left| \mu \right| =\left| \mu_{\mathsf{pp}} \right|+ \left| \mu_{ac} \right|+\left| \mu_{sc} \right|
\end{displaymath}
and (ii).

(iv) is a consequence of the definition of the Fourier trasnformability.

(v) and (vi) are obvious.

\end{proof}

\bigskip

Next, we show that if $\mu$ is a weakly admissible measure and $f \in \Cc(\hG)$ then $\left| \widehat{f} \right|^2*\mu$ defines an uniformly continuous and bounded function. This result is essential for the proof of Theorem~\ref{thm: sap as FT}. The proof of the Theorem~\ref{thm: admis implies strong admis} below follows the idea of \cite[Thm.~2.5]{ARMA1}(see also \cite[Thm.~9.18]{MoSt}, \cite[Lemma]{RobTho}).

\begin{theorem}\label{thm: admis implies strong admis} Let $\mu$ be a weakly admissible measure. Then
\begin{itemize}
\item[(i)] For each $K \subset \widehat{G}$ there exists a constant $C_K$ so that for all $f \in \Cc(\widehat{G})$ we have
\begin{displaymath}
\sqrt{\int_G \left| \widehat{f} \right|^2 d |\mu|} \leq C_K \| f \|_\infty \,.
\end{displaymath}
\item[(ii)] For each $f \in \Cc(\widehat{G})$ the function
\begin{displaymath}
t \to \int_{G} \left| \widehat{f} (x+t) \right|^2 d |\mu|(x) =: g(t)
\end{displaymath}
belongs to $\Cu(G)$.
\item[(iii)] For each $f \in \Cc(\widehat{G)}$ the function
\begin{displaymath}
t \to \int_{G} \left| \widehat{f} (x+t) \right|^2 d \mu(x)=: h(t)
\end{displaymath}
belongs to $\Cu(G)$.
\item[(iv)] $\mu$ is translation bounded.
\end{itemize}
\end{theorem}
\begin{proof}

(i) Let us start by recalling that $ L^2(|\mu|)$ is a Hilbert space with respect to the inner product
\begin{displaymath}
\langle f,g\rangle = \int_{G} f(x) \overline{g(x)} d |\mu|(x) \,.
\end{displaymath}
The norm induced by this inner product is
\begin{displaymath}
_\mu\| h \|_2:= \sqrt{ \int_{G} \left|h\right|^2 d |\mu|} \,.
\end{displaymath}
The definition of weak admissibility tells us that we can define a mapping
\begin{displaymath}
T: \Cc(\widehat{G}) \to L^2(|\mu|) \,;\, T(f)=\widehat{f} \,.
\end{displaymath}
It is obvious that $T$ is linear.

Now fix some compact set $K \subset \widehat{G}$ and define as usual
\begin{displaymath}
C(\widehat{G}:K):= \{ f \in \Cc(\widehat{G}) | \supp(f) \subset K \} \,.
\end{displaymath}
We claim that the restriction $T: C(\widehat{G}:K) \to L^2(|\mu|)$ has a closed graph and hence it is continuous.

Indeed, let $f_\alpha \to f$ in $(C(\widehat{G}:K), \| \, \|_\infty)$ be so that $\widehat{f_\alpha} \to g$ in  $L^2(|\mu|)$. We need to show that $f =g $ in $L^2(|\mu|)$.

Let $\epsilon>0$ and let $J \subset G$ be any compact set.

Since $\widehat{f_\alpha} \to g$ in  $L^2(|\mu|)$ there exists some $\beta$ so that for all $\alpha > \beta$ we have
\begin{displaymath}
\left( \int_{G} \left| \widehat{f_\alpha} - g \right|^2 d |\mu|\right)^{\frac{1}{2}} < \frac{\epsilon}{2} \,.
\end{displaymath}
Moreover, since $f_\alpha \to f$ in $(C(\widehat{G}:K), \| \, \|_\infty)$, there exists some $\gamma > \beta$ such that, for all $\alpha > \gamma$ we have
\begin{displaymath}
\| f_\alpha -f \|_\infty \leq \frac{\epsilon}{2 \theta_{\widehat{G}}(K) \sqrt{| \mu |(J)}+1} \,.
\end{displaymath}
Then, for some $\alpha >\gamma$ we have by the triangle inequality for $_\mu\| \, \|_2$
\begin{eqnarray*}
\left( \int_{J} \left| \widehat{f} - g \right|^2 d |\mu| \right)^{\frac{1}{2}} &\leq&  \left( \int_{J} \left| \widehat{f} - \widehat{f_\alpha} \right|^2 d |\mu| \right)^{\frac{1}{2}} + \left( \int_{J} \left| g - \widehat{f_\alpha} \right|^2 d |\mu| \right)^{\frac{1}{2}}  \\
   &\leq& \| \widehat{f} -\widehat{f_\alpha} \|_\infty \sqrt{| \mu |(J)}+\frac{\epsilon}{2} \\
   &=& \| f -f_\alpha \|_\infty  \theta_{\widehat{G}}(K)\sqrt{| \mu |(J)}+\frac{\epsilon}{2} \\
   &<& \epsilon \,.
\end{eqnarray*}

This shows that $\left( \int_{J} \left| \widehat{f} - g \right|^2 d |\mu| \right)^{\frac{1}{2}} < \epsilon $ for all compact sets $J \subset G$. Therefore, by the regularity of the measure $\left| \widehat{f} - g \right|^2  |\mu|$ we get
\begin{displaymath}
\left( \int_{G} \left| \widehat{f} - g \right|^2 d |\mu| \right)^{\frac{1}{2}} < \epsilon \,.
\end{displaymath}
Since $\epsilon >0$ was arbitrary, we get
\begin{displaymath}
\left( \int_{G} \left| \widehat{f} - g \right|^2 d |\mu| \right)^{\frac{1}{2}} =0 \,,
\end{displaymath}
which proves that $\widehat{f} = g$ in $L^2(|\mu|)$. Therefore, the graph of $T$ is closed and hence $T$ is continuous.

The continuity of $T$ implies the existence of $C_K$.

(ii) Fix some $K \subset \widehat{G}$ compact set so that $\supp(f) \subset K$. For the remaining of (ii), $f$ and $K$ are fixed.

For each $s \in G$ we will denote by $\phi_s$ the character on $\widehat{G}$ defined by $s$, that is
\begin{displaymath}
\phi_s(\chi):= \chi(s) \,.
\end{displaymath}
Then for all $t \in G$ we have
\begin{eqnarray*}
\int_{G} \left| \widehat{f} (x+t) \right|^2 d |\mu|(x)&=& \int_{G} \left| T_{-t} \widehat{f} (x) \right|^2 d |\mu|(x)= \int_{G} \left| \widehat{\phi_{-t} f} (x) \right|^2 d |\mu|(x) \\
&\leq& C_K \| \phi_{-t} f \|_\infty = C_K \|f \|_\infty \,.
\end{eqnarray*}
This shows that $g$ is bounded.

Next, let $\epsilon >0$. By Pontryagin duality, the set
\begin{displaymath}
N(K, \frac{\epsilon}{C_K \| f \|_\infty+1}) := \{ s \in G | \left| \psi_s( \chi) -1 \right| <\frac{\epsilon}{C_K \| f \|_\infty+1}\, \mbox{ for all } \chi \in K \}
\end{displaymath}
is an open neighbourhood of $0$ in $G$.

If $s-t \in N(K, \frac{\epsilon}{2})$, by the triangle inequality for $_\mu\| h \|_2$ we have
\begin{eqnarray*}
 \left| \sqrt{g(s)}-\sqrt{g(t)} \right| &=&\left|  _\mu\| T_{-t} \widehat{f} \|_2 -_\mu\| T_{-s} \widehat{f} \|_2 \right|\\
 &\leq& _\mu\| T_{-t} \widehat{f} - T_{-s} \widehat{f} \|_2= _\mu\|  \widehat{\phi_{-t}f} - \widehat{\phi_{-s}f} \|_2\\
 &=& _\mu\|  \reallywidehat{\phi_{-t}f- \phi_{-s}f} \|_2 \leq C_K \| \phi_{-t}f- \phi_{-s}f \|_\infty \\
 &=& C_K \| \phi_{-t}\left( 1- \phi_{t-s}\right)f \|_\infty = C_K \| \left( 1- \phi_{t-s}\right)f \|_\infty \\
 &\leq& C_K \frac{\epsilon}{C_K \| f \|_\infty+1}\| f \|_\infty < \epsilon \,.
\end{eqnarray*}
This proves that $\sqrt{g(t)}$ is uniformly continuous. Therefore, as $0 \leq g(t) \leq C_K \| f \|_\infty$, and as $x^2$ is uniformly continuous on the compact set $[ 0, \sqrt{C_K \| f \|_\infty}]$, it follows that $g$ is uniformly continuous.

(iii) Consider the decomposition
\begin{displaymath}
\mu=\mbox{Re}(\mu)+i \mbox{Im}(\mu) \,.
\end{displaymath}
of $\mu$.

Since
\begin{displaymath}
\mbox{Re}(\mu)=\frac{1}{2} (\mu+\bar{\mu}) \,,
\end{displaymath}
have
\begin{displaymath}
|\mbox{Re}(\mu)| \leq \frac{1}{2} (|\mu|+|\bar{\mu}|)=|\mu| \,.
\end{displaymath}
Same way we get
\begin{displaymath}
|\mbox{Im}(\mu)| \leq \frac{1}{2} (|\mu|+|\bar{\mu}|)=|\mu| \,.
\end{displaymath}
Therefore, by Lemma \ref{lemma-basic prop} (ii), the measures $\mbox{Re}(\mu)$ and $\mbox{Im}(\mu)$ are weakly admissible.

Next, consider the Jordan decomposition
\begin{displaymath}
\mbox{Re}(\mu)=\mbox{Re}(\mu)_+-\mbox{Re}(\mu)_- \,.
\end{displaymath}
It follows from the properties of Jordan decomposition that
\begin{displaymath}
\left| \mbox{Re}(\mu)_{\pm}\right| \leq \left| \mbox{Re}(\mu)\right|  \,.
\end{displaymath}
This shows that $ \mbox{Re}(\mu)_{\pm}$ are weakly admissible measures, and hence by (ii) the functions
\begin{displaymath}
t \to \int_{G} \left| \widehat{f} (x+t) \right|^2 d |\mbox{Re}(\mu)_{\pm}|(x)
\end{displaymath}
belong to $\Cu(G)$. As $\mbox{Re}(\mu)_{\pm} \geq 0$, we get that
\begin{displaymath}
t \to \int_{G} \left| \widehat{f} (x+t) \right|^2 d \mbox{Re}(\mu)_{\pm}(x)
\end{displaymath}
belong to $\Cu(G)$, and hence so does their difference
\begin{displaymath}
t \to \int_{G} \left| \widehat{f} (x+t) \right|^2 d \mbox{Re}(\mu)(x)\,.
\end{displaymath}
Exactly the same way, the function
\begin{displaymath}
t \to \int_{G} \left| \widehat{f} (x+t) \right|^2 d \mbox{Im}(\mu)(x)
\end{displaymath}
belongs to $\Cu(G)$.

Now, the equality
\begin{displaymath}
\int_{G} \left| \widehat{f} (x+t) \right|^2 d \mu(x)=  \int_{G} \left| \widehat{f} (x+t) \right|^2 d \mbox{Re}(\mu)(x)+i \int_{G} \left| \widehat{f} (x+t) \right|^2 d \mbox{Im}(\mu)(x)
\end{displaymath}
proves the claim.

(iv) Let $K \subset G$ be compact. Then there exists some $h \in \Cc(\widehat{G})$ so that \cite{BF,MoSt}
\begin{displaymath}
\widehat{h} \geq 1_{K} \,.
\end{displaymath}
Then, for all $x \in G$ we have
\begin{displaymath}
|\mu|(-x+K) =\int_{G} 1_{K}(x+t) d \left| \mu \right|(t) \leq  \int_{G} \left| \widehat{h} (x+t) \right|^2 d |\mu|(x) \,.
\end{displaymath}
Therefore, by (ii)
\begin{displaymath}
\| \mu \|_{K} = \sup_{x \in G} \{ |\mu|(-x+K)  \} <\infty \,.
\end{displaymath}
\end{proof}

A natural question to ask now is if translation boundedness implies weakly admissible. We will show in the next section that for $G=\RR^d$ the answer is yes, but in general the question is still open to our knowledge.

\bigskip

Next, we show that for weakly almost periodic function, weak admissibility is compatible with Eberlein decomposition.

\begin{theorem}\label{them:admis Eber decomp} Let $\mu \in \WAP$. Then, $\mu$ is weak admissible if and only if $\mu_{\mathsf{s}}$ and $\mu_0$ are weak admissible.
\end{theorem}
\begin{proof}
$\Leftarrow$ is obvious.

$\Rightarrow$. Let $f \in \Cc(\widehat{G})$. Fix some compact $K$ and pick some $g \in \Cc(G)$ such that $g \geq 1_{K}$.

Let $h := g \left| \widehat{f} \right|^2$. Then $h \in \Cc(G), h \geq 0$ and $h=\left|\widehat{f}\right|^2$ on $K$.

Finally, let $f_\alpha \in SAP(G)$ be an approximate identity for the Eberlein convolution, such that $f_\alpha \geq 0$ and $f_\alpha(-x)=f_\alpha(x)$. Then \cite[Cor.~7.2]{ARMA}
\begin{displaymath}
\mu_{\mathsf{s}}= \lim_\alpha \mu \circledast f_\alpha
\end{displaymath}
in the product topology on $\cM^\infty(G)$. In particular $\mu \circledast f_\alpha$ converges in the vague topology to $\mu_{\mathsf{s}}$.

Next, we have
\begin{equation}\label{Eq22}
\int_G h(t) d \left|\mu_{\mathsf{s}} \right|(t)   = \sup \{ \left| \int_G \phi(t) d \mu_{\mathsf{s}}(t) \right|  | \phi \in \Cc(G), |\phi| \leq h \} \,.
\end{equation}
Let $\phi \in \Cc(G)$ be so that  $|\phi| \leq h$. Then
\begin{equation}\label{Eq23}
\begin{split}
  \left| \int_G \phi(t) d \mu_{\mathsf{s}}(t) \right| &= \lim_\alpha  \left| \int_G \phi(t) d  \mu \circledast f_\alpha (t) \right| \\
   &= \lim_\alpha  \left|  \phi^\dagger *( \mu \circledast f_\alpha) (0) \right|\\
     &= \lim_\alpha  \left|  (\phi^\dagger * \mu) \circledast f_\alpha (0) \right| \,,
\end{split}
\end{equation}
with the last equality following from \cite[Thm.~6.4]{ARMA}.

Now, since $\mu$ is weak admissible, by Theorem \ref{thm: admis implies strong admis} (ii), there exists a constant $C_f$, which depends only on $f$ such that for all $t \in G$ we have
\begin{displaymath}
\int_{G} \left| \widehat{f} \right|^2(t+s) d | \mu |(s) \leq C_f \,.
\end{displaymath}
This implies that for all $t \in G$ we also have
\begin{displaymath}
\int_{G} h(t+s) d | \mu |(s) \leq C_f
\end{displaymath}
and hence,
\begin{displaymath}
\int_{G} \left|\phi \right|(t+s) d | \mu |(s) \leq C_f \,.
\end{displaymath}
Therefore, if $A_n$ is any F\"olner sequence, by the definition of Eberlein convolution, we have for all $\alpha$:
\begin{eqnarray*}
 \left|  (\phi^\dagger * \mu) \circledast f_\alpha (0) \right| &=&  \left| \lim_n \frac{1}{\vol(A_n)} \int_{A_n}(\phi^\dagger * \mu)(t) f_\alpha (-t) dt\right| \\
 &\leq&  \limsup_n \frac{1}{\vol(A_n)} \int_{A_n} \left|\phi^\dagger * \mu \right|(t) f_\alpha (t) dt \\
  &\leq&  \limsup_n \frac{1}{\vol(A_n)} \int_{A_n} \left|\int_{G} \phi (-t+s) d  \mu (s) \right|(t) f_\alpha (t) dt \\
  &\leq&  \limsup_n \frac{1}{\vol(A_n)} \int_{A_n}\bigl( \int_{G} \left|\phi \right|(-t+s) d | \mu |(s) \bigr) f_\alpha (t) dt \\
  &\leq&  \limsup_n \frac{1}{\vol(A_n)} C_f f_\alpha (t) dt \\
  &=&C_f M(f_\alpha) \,.
\end{eqnarray*}
Hence, by (\ref{Eq23}) we have
\begin{displaymath}
\left| \int_G \phi(t) d \mu_{\mathsf{s}}(t) \right| \leq \limsup_\alpha C_f M(f_\alpha)= C_f\,.
\end{displaymath}
By (\ref{Eq22}) we get
\begin{displaymath}
\int_G h(t) d \left|\mu_{\mathsf{s}} \right|(t) \leq C_f \,.
\end{displaymath}
This shows that
\begin{displaymath}
\int_{K} \left|\widehat{f}\right|^2(t) d  \left|\mu_{\mathsf{s}} \right|(t) \leq C_f \,.
\end{displaymath}
As the constant is independent of the compact set $K$, and $K \subset G$ was an arbitrary compact set, by the regularity of the measure $\left|\widehat{f}\right|^2(t) \left|\mu_{\mathsf{s}} \right|$ we get
\begin{displaymath}
\int_{G} \left|\widehat{f}\right|^2(t) d  \left|\mu_{\mathsf{s}} \right|(t) \leq C_f \,.
\end{displaymath}
This proves that $\mu_{\mathsf{s}}$ is weak admissible.

Finally $\mu_0= \mu-\mu_{\mathsf{s}}$ is weak admissible as a difference of two weak admissible measures.
\end{proof}

\section{Weak Admissible measures on $\RR^d$}

In this section we connect our concept of weak admissible with the concepts of admissibility and uniform boundedness which appeared in the work of Lin \cite{lin}, Thornett \cite{tho} and Robertson and Thornett\cite{RobTho}.

Let us first recall some of their definitions:

\begin{definition} A Borel measure $\mu$ on $\RR^d$ is called \textbf{$r$-admissible} if for all $f \in L^2(\RR^d)$ with $\supp(f) \subset [-r,r]^d$ we have
\[
\int_{\RR^d} \left|\widehat{f}(y) \right|^2 d \mu(y) < \infty
\]
\end{definition}

Of importance to us is the following Theorem, see \cite[Thm.~1]{lin}, \cite[Thm.~4.2, Thm.4.3]{tho}.

\begin{theorem} \label{thm: lin} A positive Borel measure $\mu$ on $\RR^d$ is $r$-admissible for some $r$ if and only if it is translation bounded.
In particular, a measure is $r$-admissible for some $r>0$ if and only if it is $r$-admissible for all $r >0$.
\end{theorem}

Because of this, we will simply call a measure admissible instead of $r$-admissible.

As a consequence we get the following simple characterisation of weak admissibility on $\RR^d$.

\begin{theorem} Let $\mu$ be a Radon measure on $\RR^d$. Then, the following are equivalent:
\begin{itemize}
  \item [(i)] $\mu$ is translation bounded.
  \item [(ii)] $\left| \mu \right|$ is translation bounded.
  \item [(iii)]$\left| \mu \right|$ is admissible.
  \item [(iv)] $\mu$ is weakly admissible.
   \item [(v)] For all bounded Borel sets $A\subset \RR^d$ we have $\widehat{1_A} \in L^2( \mu)$.
\end{itemize}
\end{theorem}
\begin{proof} The implication $(i) \Rightarrow (ii)$ is obvious.

$(ii) \Rightarrow (iii)$ follows from Theorem~\ref{thm: lin}.

$(iii) \Rightarrow (iv)$ is immediate as $\left| \mu \right|$ is admissible implies $\left| \mu \right|$ is weakly admissible which in turn implies $\left| \mu \right|$ is weakly admissible.

$(iv) \Rightarrow (i)$ follows from Theorem~\ref{thm: admis implies strong admis}.

The equivalence $(ii) \Leftrightarrow (v)$ is \cite[Theorem]{RobTho} applied to $|\mu |$.
\end{proof}

\section{Weak admissibility and the Fourier transform}

In this section we take a closer look at weak admissibility and Fourier trasnformability. We start by reviewing a criteria for twice Fourier transformability of a transformable measure. Next, we give a criteria for Fourier transformability of a measure $\mu$ on $\RR^d$ in terms of its Fourier transform as a tempered distribution.

\begin{theorem}\cite[Thm.3.10]{CRS2}\label{thm: twice FT} Let $\mu$ be a Fourier transformable measure. Then $\mu$ is twice Fourier transformable if and only if $\mu$ is a weak admissible measure.

In this case we have
\begin{displaymath}
\widehat{\widehat{\mu}}=\mu^\dagger \,.
\end{displaymath}
\end{theorem}

\smallskip

Next, consider a measure $\mu$ on $\RR^d$. If $\mu$ is tempered as a distribution, then $\mu$ has a Fourier transform $\psi$, which is a tempered distribution.
If $\psi$ is not a measure, it is easy to see that $\mu$ cannot be Fourier transformable in the measure sense. An interesting question is: what happens when $\psi$ is a measure?

As shown in \cite{ARMA1}, it does not necessary follows that $\mu$ is Fourier transformable as a measure.

In the following Theorem we prove that in this situation, the Fourier transformability of $\mu$ in measure sense is equivalent to the weak admissibility, and hence translation boundedness of $\psi$.

\begin{theorem}\label{thm: FT measure and distribution} Let $\mu \in \cM(\RR^d)$. Then, $\mu$ is Fourier transformable as a measure if and only if the following hold:
\begin{itemize}
\item[(i)] $\mu$ is tempered.
\item[(ii)] The Fourier transform $\nu$  of $\mu$ as a tempered distribution is a translation bounded measure.
\end{itemize}
Moreover, in this case we have
\begin{displaymath}
\widehat{\mu}=\nu \,.
\end{displaymath}
\end{theorem}
\begin{proof}

$\Rightarrow:$

Since $\mu$ is Fourier transformable, the measure $\widehat{\mu}$ is translation bounded and
\begin{displaymath}
\langle\mu, g\rangle = \langle \widehat{\mu}, \widecheck{g}\rangle
\end{displaymath}
for all $g \in \Cc(\RR^d)$. Moreover, since $\nu$ is translation bounded, it is tempered as a distribution.

As $\widehat{\mu}$ is a translation bounded measure, it is a tempered distribution. Therefore, it is the Fourier transform of a tempered distribution $v$.

Then, for all $f  \in \cS^\infty(\RR^d)$ we have
\begin{displaymath}
\langle v, g\rangle = \langle \widehat{\mu}, \widecheck{g}\rangle \,.
\end{displaymath}
This shows that for all $g \in \Cc^\infty(\RR^d)$ we have
\begin{displaymath}
\langle v,g\rangle = \langle \widehat{\mu}, \widecheck{g}\rangle = \langle\mu, g\rangle \,.
\end{displaymath}
Therefore, $v= \mu$.

As $v$ is a tempered distribution, it follows that $\mu$ is tempered as a measure, and that $\widehat{\mu}$ is the Fourier transform of $\mu$ as a tempered distribution.

As $\widehat{\mu}$ is a translation bounded measure, the claim follows.

$\Leftarrow:$ We have
\begin{displaymath}
\langle\mu, g\rangle = \langle \nu, \widecheck{g}\rangle
\end{displaymath}
for all $g \in \cS(\RR^d)$. Therefore,
\begin{displaymath}
\langle\mu, g\rangle = \langle \nu, \widecheck{g}\rangle
\end{displaymath}
for all $g \in \Cc^\infty(\RR^d)$.

Next, fix some $h_n \in \cS(\RR^d)$ such that $\| h_n \|_\infty =1, h_n=1$ on $B_{n}(0)$ and $\supp(h_n) \subset B_{n+1}(0)$. Let $g_n := \widecheck{h_n}$.

Now, pick some $f \in \Cc(G)$. Then, as $f \in L^2(\RR^d)$, we have $\widehat{f} \in L^2(\RR^d)$. Therefore, by the Lebesgue Dominated convergence Theorem,
\begin{displaymath}
\| \widehat{f} h_n- \widehat{f} \|_2 \to 0 \,.
\end{displaymath}
This shows that $f*g_n \to f$ in $L^2(\RR^d)$. Then,
\begin{displaymath}
(f*g_n)*\widetilde{(f*g_n)} \to f*\widetilde{f} \mbox{ in } (\Cc(\RR^d), \| \, \|_\infty) \,.
\end{displaymath}
This gives
\begin{displaymath}
\langle \mu, f*\widetilde{f}\rangle = \lim_n \langle \mu, (f*g_n)*\widetilde{(f*g_n)} \rangle = \lim_n \langle \nu, \left| \widecheck{f} \right|^2 h_n^2\rangle \,.
\end{displaymath}
Finally, since $\left| \widecheck{f} \right|^2 \in L^1(| \nu |)$ and $\left| \widecheck{f} \right|^2 h_n^2$ is increasing and converges pointwise to $\left| \widecheck{f} \right|^2$, we get by the monotone convergence theorem
\begin{displaymath}
\lim_n \langle \nu, \left| \widecheck{f} \right|^2 h_n\rangle= \langle \nu, \left| \widecheck{f} \right|^2 \rangle \,.
\end{displaymath}
Therefore, for all $f \in \Cc(\RR^d)$ we have $\left|\widecheck{f}\right|^2 \in L^1(| \nu |)$ and
\begin{displaymath}
\langle \mu, f*\widetilde{f}\rangle= \langle \nu, \left| \widecheck{f} \right|^2 \rangle \,.
\end{displaymath}
This shows that $\mu$ is Fourier transformable and
\begin{displaymath}
\widehat{\mu}=\nu \,.
\end{displaymath}
\end{proof}

\section{Fourier Bohr Series and Formal Sums}

Given a weakly almost periodic measure $\mu$, we can introduce its Fourier Bohr series (see Defi.~\ref{defi:FB series} below). If $\mu$ is Fourier transformable, then its Fourier Bohr series is a measure, but no guarantee that this happens in general. For this reason, when we deal with the Fourier--Bohr series of a weakly almost periodic measure, we need to threat it as a formal sum (see also \cite{ARMA}).

In this section we review the basic properties of formal sums and the Fourier Bohr series of weakly almost periodic measures.

\subsection{Formal Sums}

We start by defining the notion of formal sums.

\begin{definition} By a {\bf formal sum} we understand an expression of the form
\begin{displaymath}
\omega=\sum_{x \in G} \omega_x \delta_x \,,
\end{displaymath}
where $\omega_x \in \CC$.

For such an expression we define the {\bf support} of $\omega$ as
\begin{displaymath}
\sup(\omega):= \{ x \in G | \omega_x \neq 0 \} \,.
\end{displaymath}
\end{definition}

\begin{remark} Any formal sum is a measure on $G_{\mathsf{d}}$. Our interest will be in formal sums which are measures on $G$, so we will simply treat them as formal sums.
\end{remark}

\smallskip

We will often speak of integrals of functions against formal sums. Note that we can multiply in an obvious way a formal sum by a function, and we obtain a new formal sum. We will say that the function $f$ is integrable against the formal sum $\omega$ if the product $f \omega$ is an absolutely summable series.

\begin{definition} Let $\omega$ be a formal sum, and $f: G \to \CC$ be a function. We say that $f$ is \textbf{integrable} with respect to $\omega$ if
\begin{displaymath}
\sum_{x \in G} \left|f(x) \omega(x) \right| < \infty \,.
\end{displaymath}
In this case we define the \textbf{integral}
\begin{displaymath}
\int_G f d \omega =\langle f ,\omega\rangle := \sum_{x \in G} f(x) \omega(x) \,.
\end{displaymath}
We also denote by
\begin{displaymath}
\sL^1(\omega) := \{ f : G \to \CC | f \mbox{ is integrable with respect to } \omega \} \,.
\end{displaymath}
and by
\begin{displaymath}
\sL^2(\omega) := \{ f : G \to \CC | f^2 \mbox{ is integrable with respect to } \omega \} \,.
\end{displaymath}
\end{definition}

\begin{remark}
\begin{itemize}
\item[(i)] If $f$ is integrable with respect to $\omega$ then $f(x) \omega_x=0$ for all but at most countably many $x \in G$.
\item[(ii)] $f$ is integrable with respect to $\omega$ exactly when $\sum_{x \in G} f(x) \omega(x)$ is absolutely convergent.
\item[(iii)] If we treat $\omega$ as a measure on $\Gd$, then a function is integrable with respect to $\omega$ exactly when it is integrable with respect to $\omega$ as a measure and $\sL^1(\omega)$ is just the standard $L^1(\omega)$ space.
\end{itemize}
\end{remark}

We start by characterizing formal sums which are measures on $G$. It is clear that every such measure is pure point.

\begin{theorem}\label{thm:formal sum is a measure} Let $\omega=\sum_{x \in G} \omega_x \delta_x$ be a formal sum. Then $\omega$ is a measure if and only if for all compact sets $K$ we have
\begin{displaymath}
\sum_{x \in K} \left| \omega_x \right| < \infty \,.
\end{displaymath}
\end{theorem}
\begin{proof}

$\Rightarrow:$ Let $K$ be a compact set. Since $\omega$ is a measure, then so is its variation measure $\left| \omega \right|$.

Therefore we have
\begin{displaymath}
\sum_{x \in K} \left| \omega_x \right| = \left| \omega \right|(K) < \infty \,.
\end{displaymath}
$\Leftarrow:$ We first prove that $\Cc(G) \subset \sL^1(\omega)$.

Let $f \in \Cc(G)$, and let $K$ be any compact set such that $\sup(f) \subset K$. Then
\begin{displaymath}
\sum_{x \in G} \left|f(x) \omega(x) \right| = \sum_{x \in K } \left|f(x) \omega(x) \right| \leq  \| f \|_\infty \sum_{x \in K } \left| \omega(x) \right| < \infty \,.
\end{displaymath}
Next it is trivial to show that $\omega$ is linear on $\Cc(G)$.

Finally, if $K \subset G$ is a fixed compact set and $f \in \Cc(G)$ is so that $\sup(f) \subset K$, by the above computation we have
\begin{displaymath}
\left| \langle f,\omega\rangle \right| \leq C_K \| f \|_\infty \,,
\end{displaymath}
where
\begin{displaymath}
C_K= \sum_{x \in K} \left| \omega_x \right| < \infty \,.
\end{displaymath}
Therefore, by the Riesz representation theorem, $\omega$ is a measure.
\end{proof}

We next introduce a simpler criteria which involves a single compact set with non-empty interior.

\begin{corollary} Let $\omega=\sum_{x \in G} \omega_x \delta_x$ be a formal sum and let $K$ be a fixed compact set with a non-empty interior. Then $\omega$ is a measure if and only if for all $t \in G$ we have
\begin{displaymath}
\sum_{x \in (t+K)} \left| \omega_x \right| < \infty \,.
\end{displaymath}
\end{corollary}
\begin{proof}
Let $K' \subset G$ be any compact set. Then, since $K$ has non-empty interior, there exists $t_1,..,t_k\in G$ such that
\begin{displaymath}
K' \subset \cup_{j=1}^k t_j +K \,.
\end{displaymath}
Then
\begin{displaymath}
\sum_{x \in K} \left| \omega_x \right| \leq \sum_{j=1}^k
\left( \sum_{y \in (t_j+K)} \left| \omega_y \right| \right) <  \infty \,.
\end{displaymath}
Therefore, by Theorem \ref{thm:formal sum is a measure}, $\omega$ is a measure.
\end{proof}

\subsection{Weakly Admissible Formal Sums}

We can now extend the definition of weak admissibility to formal sums. We will see in this subsection that all weakly admissible formal sums are in fact measures.

The reason we are interested in extending the definition to formal sums is because we will be interested in weak admissibility of a Fourier Bohr series, which may or may not be a measure.

\begin{definition}\label{defi:admiss formal sums} A formal sum $\omega$ is called a \textbf{weakly admissible formal sum} if $\reallywidehat{\Kt(\widehat{G})} \subset \sL^1(\omega)$.
\end{definition}

\begin{remark}
\begin{itemize}
\item [(i)] A formal sum $\omega $ is weakly admissible if and only if for all $f \in \Cc(\widehat{G})$ we have
\begin{displaymath}
\sum_{x \in G} \left| \omega_x \right|\left| \widecheck{f} \right|^2(x) < \infty \,.
\end{displaymath}
\item [(ii)] A formal sum $\omega $ is weakly admissible if and only if $\reallywidehat{\Cc(\widehat{G})} \subset \sL^2(\omega)$.
\item[(iii)] Any formal sum which is weakly admissible is a linear function on $\reallywidehat{\Kt(\widehat{G})}$.
\end{itemize}
\end{remark}

We start by proving that weakly admissible formal sums are measures.

\begin{lemma}\label{lem: admis formal sum impies measure} Let $\omega$ be a weakly admissible formal sum. Then $\omega$ is a translation bounded measure.
\end{lemma}
\begin{proof}
Let $K \subset G$ be compact. Then there \cite{BF,MoSt} exists a function $f \in \Cc(\widehat{G})$ such that $f \geq 1_K$.

Then
\begin{displaymath}
\sum_{x \in K} \left| \omega_x \right| \leq \sum_{x \in G} \left| \omega_x \right|\left| \widecheck{f} \right|^2 < \infty \,.
\end{displaymath}
Then, by Theorem \ref{thm:formal sum is a measure} $\omega$ is a measure, which is trivially a weakly admissible measure. Hence, by Thm.~\ref{thm: admis implies strong admis}, $\omega$ is a translation bounded measure.
\end{proof}

\begin{corollary} Let $\omega$ be a formal sum on $\RR^d$. Then $\omega$ is weakly admissible if and only if $\omega$ is a translation bounded measure.
\end{corollary}

We complete the section by providing a slight generalisation to \cite[Thm.~5.5]{CRS2}, namely that translation bounded measures with Meyer set support are weakly admissible (see \cite{Meyer,Moody,NS11,RVM3} for the definition and importance of Meyer sets and cut and project schemes).  Note that this stronger version was actually proved, but not stated in \cite{CRS2}, and our proof is identical to the one from the cited result.

\begin{theorem} Let $\mu =\sum_{x \in \Lambda} \omega_x \delta_x$. If $\Lambda$ is a subset of a Meyer set and $\{ \omega_x \}$ is bounded, then
\begin{displaymath}
\omega:= \sum_{x \in \Lambda} \omega_x \delta_x
\end{displaymath}
is a weakly admissible formal sum.
\end{theorem}
\begin{proof}
Let $(G,H, \cL)$ be a cut and project scheme and $W \subset H$ a compact set such that
\begin{displaymath}
\Lambda \subset \oplam(W) \,.
\end{displaymath}

Let $(\widehat{G}, \widehat{H}, \cL^0)$ be the dual cut and project scheme. Then, there exists some $h \in \Cc(\widehat{H})$ such that $\widecheck{h} \geq 1_W$ \cite{BF,MoSt}.

Then $\omega_{h*\widetilde{h}}$ is Fourier transformable and \cite{CRS2}
\begin{displaymath}
\reallywidehat{\omega_{h*\widetilde{h}}} = \omega_{|\widecheck{h}|^2} \,.
\end{displaymath}
Therefore, as the Fourier transform of a measure, $\omega_{|\widecheck{h}|^2} $ is weakly admissible. As $|\omega| \leq  \omega_{|\widecheck{h}|^2} $, it follows from Lemma \ref{lemma-basic prop} that $|\omega|$ is weakly admissible, and hence by Lemma \ref{lemma-basic prop} $\omega$ is weakly admissible.
\end{proof}

\subsection{Fourier Bohr series}

In the spirit of \cite{ARMA} we define:

\begin{definition}\label{defi:FB series} Let $\mu \in \WAP(G)$. The \textbf{Fourier-Bohr series} of $\mu$ is defined as
\begin{displaymath}
\cF_{\mathsf{d}}(\mu):= \sum_{\chi \in G} c_\chi(\mu) \delta_\chi \,.
\end{displaymath}
\end{definition}

\medskip

As shown in \cite{ARMA}, the Fourier Bohr series uniquely identifies the strongly almost periodic component of a weakly almost periodic measure.

\begin{theorem} \cite{ARMA}
\begin{itemize}
  \item [(i)] For each $\mu \in \WAP(G)$  we have $\cF_{\mathsf{d}}(\mu) = \cF_{\mathsf{d}}(\mu_{\mathsf{s}})$.
  \item [(ii)] For $\mu, \nu \in \SAP(G)$ we have $\cF_{\mathsf{d}}(\mu)=\cF_{\mathsf{d}}(\nu)$ if and only if $\mu=\nu$.
\end{itemize}
\end{theorem}

\smallskip

Let us recall that Fourier Bohr series have the following summability property.

\begin{remark}\cite[Sect.~8]{ARMA} If $\mu \in \WAP(G)$, then, for all $g \in \Cc(G)$ we have
\begin{displaymath}
\sum_{\chi \in \widehat{G}} \left| c_\chi(\mu)  \right|^2 \left| \widehat{g}(\chi)\right|^2 < \infty \,.
\end{displaymath}
\end{remark}

\bigskip

The importance of the Fourier--Bohr series for the Fourier transform of measures is given by the following result.

\begin{theorem}\label{thm: FT is cF}  If $\mu \in \WAP(G)$ is Fourier transformable then
\begin{displaymath}
  \widehat{\mu}_{\mathsf{pp}}= \cF_{\mathsf{d}}(\mu)
\end{displaymath}
and $\cF_{\mathsf{d}}(\mu)$ is a weakly admissible formal sum.
\end{theorem}
\begin{proof}

Since $\mu$ is Fourier transformable, for all $\chi \in \widehat{G}$ we have \cite{MoSt}
\begin{displaymath}
\widehat{\mu}(\{\chi \}) = c_\chi( \mu) \,.
\end{displaymath}
Therefore
\begin{displaymath}
  \widehat{\mu}_{\mathsf{pp}}=\sum_{\chi \in G} \widehat{\mu}(\{ \chi \}) \delta_\chi=\sum_{\chi \in G} c_\chi(\mu) \delta_\chi= \cF_{\mathsf{d}}(\mu) \,.
\end{displaymath}
Moreover, $\widehat{\mu}$ is weakly admissible, and hence so is $\widehat{\mu}_{\mathsf{pp}}$.
\end{proof}

\smallskip
As an immediate consequence of Theorem~\ref{thm: FT is cF} we get:
\begin{corollary}\label{cor:sap ft admis}  If $\mu \in \SAP(G)$ is Fourier transformable then
\begin{displaymath}
  \widehat{\mu}= \cF_{\mathsf{d}}(\mu)
\end{displaymath}
and $\cF_{\mathsf{d}}(\mu)$ is a weakly admissible formal sum.
\end{corollary}

The main result in this paper, Theorem \ref{thm: main result} in next section shows that the converse of this also holds.

\bigskip

\section{Fourier Transformability of Strongly Almost Periodic Measures}

In this section we proceed to prove the main result in this paper. We then look at few consequences.

\begin{theorem}\label{thm: main result}  Let $\mu \in \SAP(G)$. Then, $\mu$ is Fourier transformable if and only if $\cF_{\mathsf{d}}(\mu)$ is a weakly admissible formal sum.

Moreover, in this case we have
\begin{displaymath}
\widehat{\mu}= \cF_{\mathsf{d}}(\mu) \,.
\end{displaymath}
\end{theorem}
\begin{proof}

$\Rightarrow$: Follows from Corollary~\ref{cor:sap ft admis}.

$\Leftarrow$:

First, since $\cF_{\mathsf{d}}$ is a weakly admissible formal sum, it is a measure by Lemma~\ref{lem: admis formal sum impies measure}. For simplicity, let us denote this measure by
\begin{displaymath}
\nu:=\cF_{\mathsf{d}}(\mu)= \sum_{\chi \in \widehat{G}} c_\chi(\mu) \delta_{\chi} \,.
\end{displaymath}
To complete the proof we show that $\nu$ satisfies the definition of the Fourier transform of $\mu$. In order to achieve this conclusion, for each $f \in \Cc(G)$ we show that $\mu*f*\widetilde{f}$ and $\reallywidecheck{\left| \widehat{f} \right|^2\nu}$ are Bohr almost periodic functions with the same Fourier-Bohr series, and hence equal. Equating them at zero gives the desired conclusion. We proceed along this line.

Let $f \in \Cc(G)$ be arbitrary. Then, as $\nu$ is a weakly admissible measure, we have $\widecheck{f} \in L^2(\nu)$. Therefore,
\begin{displaymath}
\left| \widehat{f} \right|^2\nu= \sum_{\chi \in \widehat{G}} \left| \widecheck{f} (\chi) \right|^2c_\chi(\mu) \delta_{\chi}
\end{displaymath}
is a finite measure. Let $g(x)$ denote the inverse Fourier transform of this measure. Then $g \in \Cu(G)$ is a Bohr almost periodic function \cite{EBE}, whose Fourier Bohr series is $\sum_{\chi \in \widehat{G}} \left| \widehat{f} (\chi) \right|^2c_\chi(\mu) \delta_{\chi}$.

Now since $\mu$ is translation bounded, $\mu*\left(f*\widecheck{f}\right)^\dagger \in \Cu(G)$ and \cite{ARMA,MoSt}
\begin{displaymath}
c_\chi(\mu*\left(f*\widetilde{f}\right)^\dagger) =\widecheck{f*\widetilde{f}}(\chi) c_\chi(\mu)= \left| \widecheck{f} (\chi) \right|^2c_\chi(\mu) \,.
\end{displaymath}
Therefore, $\mu*\left(f*\widetilde{f}\right)^\dagger$ has Fourier Bohr series $\sum_{\chi \in \widehat{G}} \left| \widecheck{f} (\chi) \right|^2c_\chi(\mu) \delta_{\chi}$.

It follows that the functions $\mu*\left(f*\widetilde{f}\right)^\dagger$ and $g$ are two Bohr almost periodic functions with the same Fourier Bohr series, and hence they are equal \cite{EBE,ARMA,MoSt}.

Making the expressions equal at $x=0$ we get
\begin{eqnarray*}
\langle\mu, f*\widetilde{f}\rangle &=& \mu*\left(f*\widetilde{f}\right)^\dagger(0)= g(0) \\
&=& \reallywidecheck{ \left| \check{f} \right|^2\nu } (0) = \langle \nu, \left| \widecheck{f} \right|^2 \rangle \,.
\end{eqnarray*}
This shows that for all $f \in \Cc(G)$ we have $\widecheck{f} \in L^2(\nu)$ and
\begin{displaymath}
\langle\mu, f*\widetilde{f}\rangle = \langle \nu, \left| \widecheck{f} \right|^2 \rangle \,.
\end{displaymath}
Therefore $\mu$ is Fourier transformable and
\begin{displaymath}
\widehat{\mu}=\nu=\cF_{\mathsf{d}}(\mu) \,.
\end{displaymath}
The last claim follows now from Corollary~\ref{cor:sap ft admis}.
\end{proof}

In the particular case $G=\RR^d$ we get:

\begin{theorem}Let $\mu \in \SAP(\RR^d)$. Then, $\mu$ is Fourier transformable if and only if $\cF_{\mathsf{d}}(\mu)$ is a translation bounded measure.

Moreover, in this case we have
\begin{displaymath}
\widehat{\mu}= \cF_{\mathsf{d}}(\mu) \,.
\end{displaymath}
\end{theorem}

By combining Theorem~\ref{thm: main result} with Theorem~\ref{thm: twice FT}  we get.

\begin{theorem}  Let $\mu \in \SAP(G)$. Then $\mu$ is twice Fourier transformable if and only if $\mu$ is a weakly admissible measure and $\cF_{\mathsf{d}}(\mu)$ is a weakly admissible formal sum.
\end{theorem}

Since strongly almost periodic measures are by definition translation bounded, in the particular case $G=\RR^d$ we get

\begin{corollary} \label{cor: twice ft sap} Let $\mu \in \SAP(\RR^d)$. Then $\mu$ is twice Fourier transformable if and only if $\cF_{\mathsf{d}}(\mu)$ is translation bounded measure.
\end{corollary}

\smallskip

\begin{remark} Consider the class
\begin{displaymath}
S:= \{ \mu \in \SAP(\RR^d) | \cF_{\mathsf{d}}(\mu) \mbox{ is translation bounded measure } \}\,.
\end{displaymath}
Let
\begin{displaymath}
T:= \{ \cF_{\mathsf{d}}(\mu) | \mu \in S \} \,.
\end{displaymath}
Then, by Corollary~\ref{cor: twice ft sap}, all measures in $S$ and $T$, respectively, are Fourier transformable, and the Fourier transform gives two bijection
$\hat{\, } : S \to T \,;\, \hat{\,} : T \to S$ whose composition is a reflection.
\end{remark}

\smallskip

Another immediate consequence of Theorem~\ref{thm: main result} is the following simple characterisation of Fourier transformable measures with pure point transform.

\begin{corollary} Let $\mu \in \cM^\infty(G)$. Then, the following are equivalent:
\begin{itemize}\itemsep=2pt
  \item[(i)] $\mu$ is Fourier transformable and $\widehat{\mu}$ is pure point.
  \item[(ii)] $\mu \in \SAP(G)$ and its Fourier Bohr series $\cF_{\mathsf{d}}(\mu)$ is a weakly admissible formal sum.
\end{itemize}
Moreover, in this case we have
\begin{displaymath}
\widehat{\mu}=\cF_{\mathsf{d}}(\mu) \,.
\end{displaymath}
\end{corollary}
\begin{proof}
We know that for a Fourier transformable measure $\mu$ we have $\widehat{\mu}$ is pure point if and only if $\mu \in \SAP(G)$ \cite{MoSt}.

The claim follows now from Theorem~\ref{thm: main result}.
\end{proof}

\begin{corollary} Let $\mu \in \cM^\infty(\RR^d)$. Then, the following are equivalent:
\begin{itemize}\itemsep=2pt
  \item[(i)] $\mu$ is Fourier transformable and $\widehat{\mu}$ is pure point.
  \item[(ii)] $\mu \in \SAP(\RR^d)$ and its Fourier Bohr series $\cF_{\mathsf{d}}(\mu)$ is a translation bounded measure.
\end{itemize}
Moreover, in this case we have
\begin{displaymath}
\widehat{\mu}=\cF_{\mathsf{d}}(\mu) \,.
\end{displaymath}
\end{corollary}
\begin{proof}
We know that for a Fourier transformable measure $\mu$ we have $\widehat{\mu}$ is pure point if and only if $\mu \in \SAP(G)$ \cite{MoSt}.

The claim follows now from Theorem~\ref{thm: main result}.
\end{proof}

Theorem~\ref{thm: main result} also produces the following criteria for a pure point measure to be the Fourier transform of a measure.

\begin{corollary} Let $\nu$ be a pure point measure on $\widehat{G}$. Then $\nu$ is the Fourier transform of a measure if and only if $\nu$ is weakly admissible and $\nu$ is the Fourier Bohr series of a strongly almost periodic measure.
\end{corollary}

\begin{corollary} Let $\nu$ be a pure point measure on $\RR^d$. Then $\nu$ is the Fourier transform of a measure if and only if $\nu$ is translation bounded and $\nu$ is the Fourier Bohr series of a strongly almost periodic measure.
\end{corollary}

\medskip
Theorem~\ref{thm: main result} gives an independent proof of the following result, which was proven recently in \cite{MoSt}:

\begin{theorem} Let $\mu \in \cM^\infty(G)$ be a Fourier transformable measure. Then $\mu_{\mathsf{s}}$ and $\mu_0$ are Fourier transformable and
\begin{displaymath}
\widehat{\mu}_{\mathsf{pp}}=\widehat{( \mu_{\mathsf{s}})} \,; \, \widehat{\mu}_{\mathsf{c}}=\widehat{( \mu_{0})} \,.
\end{displaymath}
\end{theorem}
\begin{proof}

Since $\mu \in \cM^\infty(G)$ is Fourier transformable, we get that $\mu \in \WAP(G)$ \cite{MoSt}. Also, $\widehat{\mu}$ is a weakly admissible measure.

Therefore, $\widehat{\mu}_{\mathsf{pp}}$ is a weakly admissible formal sum.

Moreover, we have
\begin{displaymath}
\widehat{\mu}_{\mathsf{pp}}= \cF_{\mathsf{d}}(\mu) =  \cF_{\mathsf{d}}(\mu_{\mathsf{s}}) \,.
\end{displaymath}

Therefore, $\mu_{\mathsf{s}}$ is a strongly almost periodic measure with a weakly admissible Fourier-Bohr series, and hence Fourier transformable. Moreover, its Fourier transform is
\begin{displaymath}
\widehat{\mu_{\mathsf{s}}} =  \cF_{\mathsf{d}}(\mu_{\mathsf{s}}) =\widehat{\mu}_{\mathsf{pp}} \,.
\end{displaymath}
Finally, as a difference of two Fourier transformable measures, $\mu_0=\mu-\mu_{\mathsf{s}}$ is Fourier transformable and
\begin{displaymath}
\widehat{\mu_0}=\widehat{\mu-\mu_{\mathsf{s}}}=\widehat{\mu}-\widehat{\mu_{\mathsf{s}}}=\widehat{\mu}-\left( \widehat{\mu}\right)_{\mathsf{pp}}=\left( \widehat{\mu}\right)_{\mathsf{c}} \,.
\end{displaymath}

\end{proof}

We complete the section by providing a characterisation for the class of positive definite strong almost periodic measures in terms of positivity and weak admissibility of the Fourier Bohr series.

\begin{theorem} Let $\mu \in \SAP(G)$. Then $\mu$ is positive definite if and only if  $\cF_{\mathsf{d}}(\mu)$ is a positive weakly admissible formal sum.
\end{theorem}
\begin{proof}
$\Rightarrow$: Since $\mu$ is positive definite, it is Fourier transformable and $\widehat{\mu}$ is positive \cite{ARMA1,BF,MoSt}. The claim follows now from Theorem \ref{thm: main result}.

$\Leftarrow$: By Theorem \ref{thm: main result}, $\mu$ is Fourier transformable and
\begin{displaymath}
\widehat{\mu} = \cF_{\mathsf{d}}(\mu) \geq 0 \,.
\end{displaymath}
Therefore, $\mu$ is a Fourier transformable measure with positive Fourier transform, and hence positive definite \cite{ARMA1,MoSt}.
\end{proof}

\begin{corollary} Let $\mu \in \SAP(\RR^d)$. Then $\mu$ is positive definite if and only if  $\cF_{\mathsf{d}}(\mu)$ is a positive translation bounded measure.
\end{corollary}

\section{Strongly almost periodic measures as Fourier Transforms}

In this section we provide a simple necessary and sufficient condition for a strongly almost periodic measure $\mu$ to be a Fourier transform, and list some of its consequences.

The result in Theorem~\ref{thm: sap as FT} below complements Theorem~\ref{thm: main result}. We would like to point out that if the strongly almost periodic measure $\mu$ is twice Fourier transformable, then Theorem~\ref{thm: main result} and Theorem~\ref{thm: sap as FT} become equivalent via Theorem~\ref{thm: twice FT}, but in general they are independent of each other.

\begin{theorem}\label{thm: sap as FT} Let $\mu \in \SAP(\widehat{G})$. Then the following are equivalent:
\begin{itemize}
  \item [(i)] There exists some measure $\nu$ on $G$ with $\widehat{\nu}=\mu$.
  \item [(ii)] $\cF_{\mathsf{d}}(\mu)$ is a measure, and $\mu$ is weakly admissible.
\end{itemize}
Moreover, in this case we have
\begin{displaymath}
\nu= \left(\cF_{\mathsf{d}}(\mu) \right)^\dagger \,.
\end{displaymath}
\end{theorem}
\begin{proof}

$(i) \Rightarrow (ii):$ Since $\nu$ is Fourier transformable, and $\widehat{\mu} \in \SAP(\widehat{G})$, the measure $\nu$ is pure point \cite{ARMA}.

Moreover, for all $x \in G$ we have \cite{ARMA}
\begin{displaymath}
\nu(\{ x \})= M(x \widehat{\nu})=c_{-x}(\mu) \,.
\end{displaymath}
This shows that
\begin{displaymath}
\nu= \left(\cF_{\mathsf{d}}(\mu) \right)^\dagger \,.
\end{displaymath}

Therefore, as $\nu$ is a measure, $\cF_{\mathsf{d}}(\mu)$ is a measure. Finally, as the Fourier transform of $\nu$, $\mu$ is weakly admissible.

$(ii) \Rightarrow (i):$ Define $\nu=  \left(\cF_{\mathsf{d}}(\mu) \right)^\dagger$. We claim that $\nu$ is Fourier transformable, and
\begin{displaymath}
\widehat{\nu}= \mu \,.
\end{displaymath}
Let $f \in \Cc(G)$. Then $f*\widetilde{f} \nu$ is a finite pure point measure, and hence $g=\reallywidehat{f*\widetilde{f}  \nu}$ is a strongly almost periodic function.

Moreover, by the Theorem~\ref{thm: admis implies strong admis} (iii), $\left|\widehat{f}\right|^2$ is convolvable as a function with $\mu$, and the convolution $\left|\widehat{f}\right|^2*\mu$ is continuous.

Finally, by \cite[Prop.~7.3]{ARMA} we have $\left|\widehat{f}\right|^2*\mu \in \SAP(\widehat{G})$ and the Fourier Bohr coefficients satisfy \cite[Prop.~8.2]{ARMA}
\begin{displaymath}
c_x( \left|\widehat{f}\right|^2 * \mu) = \reallywidehat{\left|\widehat{f}\right|^2}(x) c_x(\mu) =f*\widetilde{f}(-x) c_x(\mu)= f*\widetilde{f}(-x) \nu(\{ -x \}) \,.
\end{displaymath}
As $g$ is the Fourier transform of the finite pure point measure $\reallywidehat{f*\widetilde{f} \nu}$, it is also strongly almost periodic as measure and \cite{ARMA,MoSt}
\begin{displaymath}
c_x(g)= f*\widetilde{f}(-x) \nu(\{ -x \}) \,.
\end{displaymath}
This shows that $g$ and $\left|\widehat{f}\right|^2*\mu$ are two strongly almost periodic measures which have the same Fourier Bohr series, therefore they are equal. We also know that $g \in \Cu(G)$ and, by Theorem~\ref{thm: admis implies strong admis} (iii) we have $\left|\widehat{f}\right|^2*\mu \in \Cu(G)$. It follows that $g= \left|\widehat{f}\right|^2*\mu$ as functions. In particular
\begin{displaymath}
\langle f*\widetilde{f} , \nu \rangle = g(0)=  \left|\widehat{f}\right|^2*\mu(0)= \langle \mu, \left|\widecheck{f}\right|^2 \rangle \,.
\end{displaymath}
Hence, by the weak admissibility of $\mu$ for all $f \in \Cc(G)$ we have $\left|\widecheck{f}\right|^2 \in L^1(|\mu|)$ and
\begin{displaymath}
\langle f*\widetilde{f} , \nu \rangle = \langle \mu, \left|\widecheck{f}\right|^2 \rangle \,.
\end{displaymath}
Therefore, by the definition of Fourier transformability, $\nu$ is Fourier transformable and
\begin{displaymath}
\widehat{\nu}=\mu \,.
\end{displaymath}
\end{proof}

As above, when $G=\RR^d$ we get

\begin{theorem} Let $\mu \in \SAP(\widehat{\RR^d})$. Then $\mu$ is the Fourier transform of a measure if and only if $\cF_{\mathsf{d}}(\mu)$ is a measure. Moreover, in this case we have
\begin{displaymath}
\reallywidehat{ \left(\cF_{\mathsf{d}}(\mu) \right)^\dagger}=\mu \,.
\end{displaymath}
\end{theorem}

As a consequence of Theorem~\ref{thm: sap as FT} we also get a new proof of the following result.
\begin{theorem} \cite[Thm.~11.2]{ARMA} Let $\mu$ be a Fourier transformable measure. Then $\mu_{\mathsf{pp}}, \mu_{\mathsf{c}}$ are Fourier transformable and
\begin{displaymath}
\widehat{(\mu)_{\mathsf{pp}}}=(\widehat{\mu})_{\mathsf{s}} \, \quad \, \widehat{(\mu)_{\mathsf{c}}}=(\widehat{\mu})_{0} \,.
\end{displaymath}
\end{theorem}
\begin{proof}

Since $\mu$ is Fourier transformable, $\widehat{\mu} \in \WAP(G)$ \cite{ARMA} is weakly admissible. Then, by Theorem~\ref{them:admis Eber decomp} $(\widehat{\mu})_{\mathsf{s}}$ is weakly admissible.

Moreover, we have \cite[Thm.~11.3]{ARMA} or \cite{MoSt}
\begin{displaymath}
c_\chi(\widehat{\mu})=\mu(\{ -x \})
\end{displaymath}
which shows that $\cF_{\mathsf{d}}(\widehat{\mu})=\left( \mu_{\mathsf{pp}} \right)^\dagger$, and hence $\cF_{\mathsf{d}}(\widehat{\mu})$ is a measure.

Therefore, by Theorem~\ref{thm: sap as FT}, the measure $\cF_{\mathsf{d}}(\widehat{\mu})^\dagger=  \mu_{\mathsf{pp}}$ is Fourier transformable and
\begin{displaymath}
\widehat{(\mu)_{\mathsf{pp}}}=(\widehat{\mu})_{\mathsf{s}} \,.
\end{displaymath}
By taking differences, it follows that $\mu_{\mathsf{c}}$ is also Fourier transformable and
\begin{displaymath}
\widehat{(\mu)_{\mathsf{c}}}=(\widehat{\mu})_{0} \,.
\end{displaymath}
\end{proof}

\section{On a special class of cut and project formal sums}

In this section we review a large class of strongly almost periodic measures, and discuss their Fourier transformability.

Consider a cut and project scheme $(G, H, \cL)$, for $h \in \Cz(H)$ we define the formal sum
\[
\omega_h: \sum_{(x, x^\star) \in \cL} h(x^\star) \delta_x \,.
\]

The following Lemma is trivial, see \cite{BM,NS11}.

\begin{lemma} If $h \in \Cc(H)$ then $\omega_h$ is strongly almost periodic.
\end{lemma}

We next calculate the Fourier-Bohr series of this measure. Computations like this have been made in many places before \cite{LR,NS11,CR, CRS2}.

\begin{lemma}  If $h \in \Cc(H)$ then
\[
\cF_{\mathsf{d}}(\omega_h) = \dens(\cL)\omega_{\check{h}} \,.
\]
\end{lemma}
\begin{proof}
The computation is standard:

Let $\chi \in \widehat{G}$. Then, by \cite[Thm.~9.1]{LR} we have
\[
c_\chi(\omega_h)=\dens(\cL) \int_{H} \chi^\star(t) h(t) dt =\dens(\cL) \widecheck{f}(\chi^\star) \,.
\]
\end{proof}

Also, let us recall the following result:
\begin{theorem}\label{them: L1} \cite{CRS} If $\omega_h$ is a translation bounded measure then $h \in L^1(H)$.
\end{theorem}

We are now ready to prove the following result, compare \cite{CRS}:

\begin{theorem}\label{thm:omegah FT} \cite{CRS} Let $(G, H, \cL)$ be a cut and project scheme and let $h \in \Cc(H)$. Then, the following are equivalent:
\begin{itemize}
\item[(i)] $\omega_h$ is Fourier transformable.
\item[(ii)]$\omega_{\check{h}}$ is a weakly admissible formal sum.
\item[(iii)]$\omega_{\check{h}}$ is a translation bounded measure.
\item[(iv)] $\check{h} \in L^1(\widehat{H})$.
\end{itemize}
\end{theorem}
\begin{proof}

The equivalence $(i) \Leftrightarrow (ii)$ follows from Theorem~\ref{thm: main result}.

$(ii) \Rightarrow (iii)$ follows from Theorem~\ref{thm: admis implies strong admis}, while $(iii) \Rightarrow (iv)$ follows from Theorem~\ref{them: L1}.

We give here a second alternate proof, based on weak admissibility.

$(i) \Rightarrow (ii)$. Since $h \in \Cc(G)$, the measure $\omega_h$ is strongly almost periodic, and hence $\widehat{\omega_h}$ is pure point \cite{MoSt}.

$(iv) \Rightarrow (ii)$

Let $g \in K_2(G)$. Then $g \odot h \in \Cc(G \times H)$ and $\widehat{g \odot h} \in L^1(G \times H)$, and hence \cite{ARMA1,ReiterSte}, $\widecheck{g \odot h} \in L^1(\delta_{\cL^0})$. This is equivalent to
\[
\left| \omega_{\check{h}} \right| (\left| \widecheck{g} \right|) < \infty \,,
\]
which gives the $\widehat{K_2(G)}$-boundedness.

\end{proof}

\begin{remark} If $h \in \Cc(H)$ and $\check{h} \notin L^1(\widehat{H})$, the it follows that $\omega_h \in \SAP(G)$ but $\omega_h$ is not Fourier transformable as a measure.

This provides many examples of non Fourier transformable strongly almost periodic measures. In particular, for all these measures, the Fourier-Bohr series is not weakly admissible.
\end{remark}

We complete the section by recalling a result of \cite{NS11}. This result, together with Theorem~\ref{thm:omegah FT} provides a characterisation for Fourier transformability for strongly almost periodic measures supported inside Meyer sets.

\begin{theorem}\label{thm: NS11} Let $\omega$ be a translation bounded measure with Meyer set support. Then $\omega$ is strongly almost periodic if and only if there exists a cut and project scheme $(G, H, \cL)$ and a function $h \in \Cc(H)$ such that
\[
\omega= \omega_h \,.
\]
\end{theorem}

As a consequence we get

\begin{theorem} Let $\omega$ be a strongly almost periodic measure with Meyer set support. Then, the following are equivalent:
\begin{itemize}
  \item [(i)] $\omega$ is Fourier transformable.
  \item [(ii)] There exists a cut and project scheme $(G, H, \cL)$ and a function $h \in \Cc(H)$  with $\widehat{h} \in L^1(\widehat{H})$ such that
      \[
      \omega=\omega_h \,.
      \]
  \item [(iii)] For each cut and project scheme $(G, H, \cL)$ and function $h \in \Cc(H)$  such that
      \[
      \omega=\omega_h \,,
      \]
      we have $\widehat{h} \in L^1(\widehat{H})$.
\end{itemize}
\end{theorem}
\begin{proof}

$(i) \Rightarrow (ii):$ Theorem~\ref{thm: NS11} gives the existence of the cut and project scheme. Now, since $\omega$ is Fourier transformable, by Theorem~\ref{thm:omegah FT} we get $\widehat{h} \in L^1(\widehat{H})$.

$(ii) \Rightarrow (i):$ Follows from Theorem~\ref{thm:omegah FT}.

$(i) \Rightarrow (iii):$ Follows from Theorem~\ref{thm:omegah FT}.

$(iii) \Rightarrow (i):$ Theorem~\ref{thm: NS11} gives that there exists a cut and project scheme and some $h \in \Cc(H)$ such that $\omega=\omega_h$. Now, by (iii), we have $\widehat{h} \in L^1(\widehat{H})$. (i) follows now from Theorem~\ref{thm:omegah FT}.

\end{proof}

\subsection*{Acknowledgments} This paper was inspired by some discussions with Christoph Richard and Michael Baake, and we are grateful for the insight they provided. We are also grateful to Jeet Trivedi for carefully reading the manuscript and suggesting some improvements. We would also like to thank NSERC for their support with grant 03762-2014. This work was presented at the conference \textit{Workshop on Aperiodic Order: Enumeration Problems, Dynamics, and Toppology} at University of Bielefeld, which was supported by the German Research Foundation (DFG), within the CRC 701, and the author would like to thank the organizers.

\end{document}